\documentclass[smallextended]{svjour3}
\usepackage[top=20truemm,bottom=30truemm,left=30truemm,right=30truemm]{geometry}
\usepackage{amsmath,amssymb}
\usepackage{bm}
\usepackage{graphicx}
\usepackage[caption=false]{subfig}
\usepackage{verbatim}
\usepackage{wrapfig}
\usepackage{ascmac}
\usepackage{makeidx}
\usepackage{multirow}
\usepackage{amsfonts}
\usepackage{enumerate}
\usepackage{mathrsfs}

\usepackage{amsthm}
\theoremstyle{plain}
\spnewtheorem{thm}{Theorem}{\bf}{\it}
\spnewtheorem{prop}{Proposition}[section]{\bf}{\it}
\spnewtheorem{lem}[prop]{Lemma}{\bf}{\it}
\spnewtheorem{cor}[prop]{Corollary}{\bf}{\it}
\spnewtheorem{asm}{Assumption}{\bf}{\it}

\theoremstyle{remark}
\spnewtheorem{rem}[prop]{Remark}{\it}{}

\makeatletter
\renewenvironment{proof}[1][\proofname]{\par\pushQED{\qed}
  \normalfont
  \topsep6\p@\@plus6\p@ \trivlist
  \item[\hskip\labelsep{\bfseries #1}\@addpunct{\bfseries}]\ignorespaces
}{%
  \popQED\endtrivlist\@endpefalse
}
\renewcommand{\proofname}{Proof.}
\makeatother

\newcommand{\abs}[1]{\left\lvert#1\right\rvert}
\newcommand{\norm}[1]{\left\lVert#1\right\rVert}
\newcommand{\R}{\mathbb{R}}
\newcommand{\E}{\mathcal{E}}
\newcommand{\T}{\mathcal{T}}
\newcommand{\Ei}{\E^{\circ}}
\newcommand{\Eb}{\E^{\partial}}
\newcommand{\mean}[1]{\ensuremath{\{\!\!\{#1\}\!\!\}}}
\newcommand{\jump}[1]{\ensuremath{[\![#1]\!]}}
\newcommand{\coloneqq}{\mathrel{\mathop:}=}%
\newcommand{\pderiv}[3][]{\frac{\partial^{#1}#2}{\partial #3^{#1}}}
\newcommand{\diam}{\operatorname{diam}}

\newcommand{\supp}{\operatorname{supp}}
\newcommand{\dist}{\operatorname{dist}}
\usepackage{xcolor}
\definecolor{cyan20}{cmyk}{.2,0,0,0}

\makeatletter
    
    \@addtoreset{equation}{section}
  \makeatother
\journalname{}

\begin{document}
\title{Weak discrete maximum principle and $L^\infty$ analysis of the DG method for the Poisson equation on a polygonal domain}
\titlerunning{WMP and $L^\infty$ analysis of the DG method} 
\author{
Yuki Chiba \and Norikazu Saito 
}
\institute {
Y. Chiba \at Graduate School of Mathematical Sciences, The University of
Tokyo, Komaba 3-8-1, Meguro, Tokyo 153-8914, Japan.\\\email{ychiba@ms.u-tokyo.ac.jp}
\and 
N. Saito \at Graduate School of Mathematical Sciences, The University of
Tokyo, Komaba 3-8-1, Meguro, Tokyo 153-8914, Japan.\\\email{norikazu@g.ecc.u-tokyo.ac.jp}
}
\maketitle
%
%
 \begin{abstract}
We derive several $L^\infty$ error estimates for 
  the symmetric interior penalty (SIP) discontinuous Galerkin (DG)
  method applied to the Poisson equation in a two-dimensional polygonal
  domain. Both local and global estimates are examined. The weak maximum principle (WMP) for the discrete harmonic function is also
  established. We prove our $L^\infty$ estimates using this WMP and
  several $W^{2,p}$ and $W^{1,1}$ estimates for the Poisson
  equation. Numerical examples to validate our results are also presented. 
  \end{abstract}

\section{Introduction}

The discontinuous Galerkin (DG) method, which was proposed originally by Reed and Hill \cite{osti_4491151} in 1973, is a powerful method for solving numerically a wide range of partial differential equations (PDEs). We use a discontinuous function which is a polynomial on each element and introduce the numerical flux on each element boundary. The DG scheme is then derived by controlling the numerical flux to ensure the local conservation law. 

For linear elliptic problems, the study of stability and convergence developed well in the early 2000s; see the standard references 
\cite{MR1885715}, 
\cite{riv08} 
and \cite{MR2485457} 
for the detail. However, most of those works are based on the $L^2$ and
DG energy norms, and a little is done using other norms. An exception
is \cite{MR2113680}, where the optimal order error estimate in the
$L^\infty$ norm was proved using the discrete Green function. However,
the DG scheme in \cite{MR2113680} is defined in the exactly-fitted
triangulation of a smooth domain; this restriction is somewhat
unrealistic for practical applications. From the view point of applications, the $L^p$ theory, especially the $L^\infty$ theory, plays an important role in the analysis of nonlinear problems. Therefore, the development of the $L^p$ theory for the DG method is a subject of great importance. 

Another important subject for confirming the validity of numerical methods is the discrete maximum principle (DMP). 
Nevertheless, only a few works has been devoted to DMP for DG
method. Horv\'{a}th and Mincsovics \cite{MR3015392} proved DMP for DG
method applied to the one-dimensional Poisson equation. Badia, Bonilla
and Hierro (\cite{MR3315069}, \cite{MR3646366}) proposed nonlinear DG
schemes satisfying DMP for linear convection-diffusion equations in the one and two space dimensions. 
However, to the best of our knowledge, no results are known regarding DMP for the standard DG method in the higher-dimensional space domain. 
  
In contrast, the $L^p$ theory and DMP have been actively studied regarding the
standard finite element method (FEM). 
The pioneering work by Ciarlet and Raviart \cite{MR0375802} studied $L^\infty$ and $W^{1,p}$ error estimates together with DMP; in particular, those error estimates were proved as a consequence of DMP. Then, the $L^\infty$ error estimates were proved using several methods; Scott \cite{MR0436617} applied the discrete Green function and 
Nitsche \cite{MR568857} utilized the weighted norm technique. 
Rannacher and Scott succeeded in deriving 
the optimal $W^{1,p}$ error estimate for $1\le p\le \infty$ in 
\cite{MR645661}. 
Detailed local and global pointwise estimates have been studied by many researchers, such as \cite{MR1464148}. 
Recently, the optimal order $W^{1,\infty}$ and $L^\infty$ stability and
error estimates were established for the Poisson equation defined in a
smooth domain; the effect of polyhedral approximations of a smooth
domain was precisely examined. See \cite{2018arXiv180400390K} for the detail.  

As is well known, the non-negativity assumption (non-obtuse assumption)
on the triangulation is necessary for DMP to hold in the standard
FEM. In this connection,  
Schatz \cite{MR551291} is noteworthy in this area for deriving the \emph{weak
maximum principle} (WMP) without the non-negativity assumption and applying it to the proof of the stability, local and global error estimates in the $L^\infty$ norm. 

In this paper, we are motivated by \cite{MR551291}, and extend the results of that study to the symmetric interior penalty (SIP) DG method which is one of the popular DG method for the Poisson equation. Our results are summarized as follows.  
Let $u$ be the solution of the Dirichlet boundary value problem for the
Poisson equation \eqref{eq:poisson} defined in a polygonal domain
$\Omega\subset\mathbb{R}^2$, and let $u_h$ be the solution of the
SIPDG method \eqref{eq:dg} in the finite dimensional space $V_h$ defined 
as \eqref{eq:vh}. Then, we have the $L^\infty$ interior error estimate
(see Theorem \ref{thm1})
\[
 \norm{u-u_h}_{L^\infty(\Omega_0)} \le C \left(\inf_{\chi \in V_h}\norm{u-\chi}_{\alpha(h),\Omega_1}+\norm{u-u_h}_{L^2(\Omega_1)}\right),
\]
where $\Omega_0$ and $\Omega_1$ are open subsets such that $\Omega_0
\subset \Omega_1 \subset \Omega$, and $C$ is independent of $h$ and the
choice of $\Omega_0$ and $\Omega_1$. This interior estimate is
valid under Assumption \ref{asm1} below, where $\alpha(h)$ and
$\|\cdot\|_{\alpha(h),\Omega_1}$ are defined. 
Using this interior error estimate, we prove
the WMP (see Theorem \ref{thm2})
\[
 \norm{u_h}_{L^\infty(\Omega)} \le C \norm{u_h}_{ L^\infty(\partial
 \Omega)}
\]
for the discrete harmonic function $u_h\in V_h$.      

Finally, under some assumptions on the triangulation (see Assumption
\ref{asm2}), we prove the $L^\infty$ error estimate
(see Theorem \ref{thm3}) 
\[
 \norm{u-u_h}_{L^\infty(\Omega)} \le C \left(\inf_{\chi \in V_h}\norm{u-\chi}_{\alpha(h),\Omega} + \norm{u-u_h}_{L^\infty(\partial \Omega)}\right) 
\]
for the solution $u$ of the Poisson equation \eqref{eq:poisson} and its
DG approximation $u_h$. As a matter of fact, the WMP is a key point of
the proof of this error estimate. 
Moreover, we obtain the $L^\infty$ error estimate of the form (see Corollary \ref{cor2})
    \[
       \norm{u-u_h}_{L^\infty(\Omega)} \le 
C h^{r}\norm{u}_{W^{1+r,\infty}(\Omega)},
    \]
where $r$ denotes the degree of approximate polynomials. 
Unfortunately, this error estimate is only sub-optimal even for the
piecewise linear element ($r=1$). This is because the Dirichlet boundary
condition is imposed ``weakly'' by the variational formulation (Nitsche'
method) in
the DG method. On the other hand, it is imposed ``strongly'' by the nodal
interpolation in the FEM. This implies that we further need to more deeply consider the imposition of the Dirichlet boundary condition
in the DG method. In particular, study of better, more precise estimates of
$\alpha(h)$ and $\norm{u-u_h}_{L^\infty(\partial \Omega)}$ are necessary, and will be a focus of our future works.

Our SIPDG scheme and main results, Theorems \ref{thm1}, \ref{thm2} and \ref{thm3}, are
presented in Section \ref{result}. 
The main tool of our analysis is several $W^{2,p}$ and $W^{1,1}$
estimates for solutions of the Poisson equation. Several local error
estimates developed in previous works (see \cite{MR0520174},
\cite{MR2113680} and \cite{MR0431753}) are also used. 
After having presented those preliminary results in Section \ref{sect2}, we state the proofs of
Theorems \ref{thm1}, \ref{thm2} and \ref{thm3} in Sections
\ref{s:I},  
\ref{s:II} and   
\ref{s:III}, respectively.  
Finally, we report the results of numerical
experiments to confirm the validity of our theoretical results in Section \ref{s:ne}. 

Before concluding this Introduction, here, we list the notation used in this
paper. 

\paragraph{Notation.} 
We follow the standard notation, for example, of \cite{MR2424078} as for function spaces and their norms. In particular, for $1\le p \le \infty$ and a positive integer $j$, we use the standard Lebesgue space $L^{p}(\mathcal{O})$ and the Sobolev space $W^{j,p}(\mathcal{O})$. Here and hereinafter, $\mathcal{O}$ denotes a bounded domain in $\R^2$. The semi-norms and norm of $W^{j,p}(\mathcal{O})$
 are denoted, respectively, by 
\[
    \abs{v}_{W^{i,p}(\mathcal{O})} = \left(\sum_{\abs{\alpha} = i}\norm{\pderiv[\alpha]{v}{x}}_{L^p(\mathcal{O})}^p\right)^{{1}/{p}},\quad 
    \norm{v}_{W^{j,p}(\mathcal{O})} = \left(\sum_{i = 0}^{j} \abs{v}_{W^{i,p}(\mathcal{O})}^p\right)^{{1}/{p}}.
\]
Letting $C_0^\infty(\mathcal{O})$ be the set of all infinity differentiable functions with compact support in $\mathcal{O}$, $W^{j,p}_0(\mathcal{O})$ denotes the closure of $C^\infty_0(\mathcal{O})$ in the $W^{j,p}(\mathcal{O})$ norm. The space $W^{1,p}_0(\mathcal{O})$ is characterized as $W^{j,p}_0(\mathcal{O})=\{v\in W^{1,p}(\mathcal{O})\colon v|_{\partial\mathcal{O}=0}\}$ if $\partial\mathcal{O}$ is a Lipschitz boundary. 
Let $p'$ be the H\"older conjugate exponent of $p$; $1/p+1/p'=1$. The inner product of $L^2(\mathcal{O})$ is denoted by $(\cdot,\cdot)_{\mathcal{O}}$. The $\R^d$-Lebesgue measure of $\mathcal{O}$ is denoted by $\abs{\mathcal{O}}_d$. We also use the fractional order Sobolev space $W^{s,p}(\mathcal{O})$ for $s>0$. As usual, we write as $H^s(\mathcal{O}) = W^{s,2}(\mathcal{O})$. 
For $\Gamma \subset \partial \mathcal{O}$, we define $W^{j,p}(\Gamma)$ and $H^s(\Gamma)$ using a surface measure $ds=ds_\Gamma$.

For $\mathcal{O}_1,\mathcal{O}_2\subset\R^2$, we write 
$\mathcal{O}_1\Subset\mathcal{O}_2$ and 
$\mathcal{O}_2\Supset\mathcal{O}_1$ to express $\mathcal{O}_1\subset\mathcal{O}_2$ and 
$\overline{\mathcal{O}_1}\subset\mathcal{O}_2$. 

Finally, the letter $C$ denotes a generic positive constant depending only on $\Omega$, $\partial\Omega$, the criterion $\sigma_0$ of the penalty parameter and the shape-regularity constant $C_*$ defined in Section \ref{result}.

\section{DG scheme and results}
\label{result}

Throughout this paper, letting $\Omega$ be a bounded polygonal domain in $\R^2$, we consider the Dirichlet boundary value problem for the Poisson equation
\begin{equation}
    \left\{
    \begin{array}{rcc}
    -\Delta u = f & \text{in} & \Omega \\
    u = g & \text{on} & \partial \Omega,
    \end{array}\right.\label{eq:poisson}
\end{equation}
where $f \in L^2(\Omega)$ and $g \in H^{1/2}(\partial\Omega)$. There exists an extension $\tilde{g}\in H^1(\Omega)$ such that $\tilde{g}=g$ on $\partial\Omega$ and $\|\tilde{g}\|_{H^1(\Omega)}\le C\|g\|_{H^{1/2}(\partial\Omega)}$. The following discussion does not 
depend on the way of extension. 

Then, the weak formulation of \eqref{eq:poisson} is stated as follows. 

\smallskip

\textup{(BVP;$f,g$)} Find $u=u_0+\tilde{g}\in H^1(\Omega)$ and $u_0\in H^1_0(\Omega)$ such that 
\begin{equation}
    \int_\Omega \nabla u_0 \cdot \nabla v ~dx = (f,v)_\Omega-\int_\Omega \nabla \tilde{g} \cdot \nabla v ~dx \quad {}^\forall v \in H^1_0(\Omega). \label{eq:weakhomodiriclet}
\end{equation}

The Lax--Milgram theory guarantees that \textup{(BVP;$f,g$)} admits a unique solution $u\in H^1(\Omega)$. The regularity of $u$ is well studied. See \cite[Theorem 1]{MR0466912} and \cite[Lemma 1.2]{MR551291} for the detail of the following results. 

\begin{prop}\label{prop:poisson}
    Let $0 < \alpha < 2\pi$ be the maximum (interior) angle of $\partial\Omega$, and set $\beta = \pi/\alpha$. Letting $g=0$, we suppose that $u \in H^1_0(\Omega)$ is the solution of \textup{(BVP;$f,0$)}. 
    \begin{enumerate}
    \item[\textup{(i)}] If $\Omega$ is convex ($\beta > 1$), then $u$ belongs to $H^2(\Omega) \cap H^1_0(\Omega)$, and we have
    \begin{equation}
    \abs{u}_{H^2(\Omega)} \le \norm{f}_{L^2(\Omega)}. \label{eq:convexregulality}
    \end{equation}
    \item[\textup{(ii)}] If $1/2 < \beta < 1$ and $f \in L^p(\Omega)$ for some $1<p<2/(2-\beta)$, then $u$ belongs to $ W^{2,p}(\Omega) \cap H^1_0(\Omega)$, and there exists a positive constant $C$ depending only on $\Omega$ and $p$ such that 
    \begin{equation}
    \norm{u}_{W^{2,p}(\Omega)} \le C \norm{f}_{L^p(\Omega)}. \label{eq:nonconvexregulality}
    \end{equation}
\end{enumerate}
\end{prop}

\begin{rem}
Because $\beta > 1/2$, we have $2/(2-\beta) > 4/3$ and, consequently, \eqref{eq:nonconvexregulality} holds for some $4/3 < p \le 2$.
\end{rem}

Let $\{\T_h\}_h$ be a family of shape-regular and quasi-uniform triangulations of $\Omega$ (see \cite[(4.4.15), (4.4.16)]{bs08}). That is, there exists a positive constant $C_*$ satisfying
\begin{equation}
    \frac{h_K}{\rho_K} \le C_* ,\, \,\frac{h}{h_K} \le C_*\qquad (K\in\T_h\in\{\T_h\}_h). \label{eq:triangulation}
\end{equation}
Therein, $h_K$ and $\rho_K$ denote the diameters of the circumscribed and inscribed circles of $K$, respectively. Moreover, the granularity parameter $h$ is defined as $h =  \displaystyle \max_{K \in \T_h}h_K$. 
We set, for $1\le p\le \infty$, 
\begin{equation}
V^p  = V^p(\Omega) \coloneqq \{ v \in L^p(\Omega) \colon v|_K \in W^{1,p}(K) , (\nabla v)|_{\partial K} \in L^{p}(\partial K) {}^\forall K \in \T_h\}. 
\end{equation}
It is noteworthy that $v$ is a continuous function on $K \in \T_h$ if $v \in V^\infty$.
For a positive integer $r$, we define the finite element spaces $V_h$ and $\mathring V_h$ as 
\begin{align}
V_h &= V_h^r(\Omega) \coloneqq\{ v_h \in L^2(\Omega) \colon v_h|_K
 \in\mathcal{P}^r(K) \text{ for each }K \in \T_h\}, \label{eq:vh}\\
\mathring V_h & = \mathring V_h(\Omega) \coloneqq \{v_h \in V_h \colon
 \supp{v_h} \subset \Omega\},\label{eq:vho}
\end{align}
where $\mathcal{P}^r(K)$ denotes the set of all polynomials of degree $\le r$. 
For $\mathcal{O}\subset\Omega$, we set 
$V^p(\mathcal{O})=\{v|_{\mathcal{O}}\colon v\in V^p\}$, $V_h(\mathcal{O})=\{v|_{\mathcal{O}}\colon v\in V_h\}$ and 
$\mathring V_h(\mathcal{O}) \coloneqq \{v_h \in V_h(\mathcal{O}) \colon \supp{v_h} \subset \mathcal{O}\}$.

We let $\E_h$ be the set of all edges of $K \in \T_h$, and set 
\[
    \Eb_h \coloneqq \{e \in \E_h \colon e \subset \partial \Omega \},\quad \Ei_h \coloneqq \E_h \setminus \Eb_h.
\]
For $v \in V^p$ and $e \in \E_h$, we define $\mean{\cdot}$ and $\jump{\cdot}$ as follows. 
If $e \in \Ei_h$, we set 
\begin{align}
\mean{v} \coloneqq \frac{1}{2}(v_1 +v_2)\,,\,\,&
\jump{v} \coloneqq v_1n_1 + v_2n_2 \,,\,\\
\mean{\nabla v} \coloneqq \frac{1}{2}(\nabla v_1 + \nabla v_2)\,,\,&
\jump{\nabla v} \coloneqq \nabla v_1 \cdot n_1 + \nabla v_2 \cdot  n_2  \,.
\end{align}
If $e \in \Eb_h$, we set 
\begin{align}
\mean{v} \coloneqq v \,,\,\jump{v} \coloneqq v n\,,\,\mean{\nabla v} \coloneqq \nabla v \,,\,\jump{\nabla v} \coloneqq \nabla v \cdot n \,.
\end{align}
Therein, for $e \in \Ei_h$, there exist distinct $K_1,\,K_2 \in \T_h$ satisfying $e \subset \partial K_1 \cap \partial K_2$ and $v_i = v|_{K_i}$, where 
$n_i$ denotes the outward unit normal vector to $e$ of $K_i$, and $n$ denotes the outward unit normal vector on $\partial \Omega$. 

We define norm $\norm{v}_{V^p(\mathcal{O})}$ on $V^p(\mathcal{O})$, where $\mathcal{O}=\Omega$ or $\mathcal{O}\subset\Omega$, as 
\[
 \norm{v}_{V^p(\mathcal{O})}^p \coloneqq
\sum_{K \in \T_h} \norm{v}_{W^{1,p}(K \cap \mathcal{O})}^p  + \sum_{e \in \E_h} h_e^{1-p}\norm{\jump{v}}_{L^p(e \cap \overline{\mathcal{O}})}^p 
+ \sum_{e \in \E_h} h_e \norm{\mean{\nabla v}}_{L^p(e \cap \overline{\mathcal{O}})}^p 
\]
and 
\[
 \norm{v}_{V^\infty(\mathcal{O})} \coloneqq \max_{K \in \T_h} \norm{v}_{W^{1,\infty}(K \cap \mathcal{O})} 
 + \max_{e \in \E_h}h_e^{-1}\norm{\jump{v}}_{L^\infty(e\cap\overline{\mathcal{O}})}+ \max_{e \in \E_h} \norm{\mean{\nabla v}}_{L^\infty(e\cap\overline{\mathcal{O}})},
\]
where $h_e = (h_{K_1} + h_{K_2})/2$ if $e \in \Ei_h$ and $h_e = h_K$ if $e \in  \Eb_h$.

Letting $1 \le p ,p'\le \infty$ and $1/p+1/p'=1$, we introduce the DG bilinear form on $V^p\times V^{p'}$ as
\begin{multline}
a(u,v)  \coloneqq \sum_{K \in \T_h} \int_K \nabla u \cdot \nabla v ~dx \\
-\sum_{e \in \E_h} \int_e(\mean{\nabla u}\jump{v} + \mean{\nabla v}\jump{u})~ ds 
 + \sum_{e \in \E_h}\frac{\sigma}{h_e}\int_e \jump{u}\jump{v} ~ ds \label{eq:bilinear}
\end{multline}
for $u \in V^p$ and $v \in V^{p'}$. Herein, $\sigma$ is a sufficiently large constant. 
The linear form $F$ on $V^2$ is defined by
\begin{equation}
    F(v) \coloneqq \int_\Omega fv ~ dx + \sum_{e \in \Eb_h}\int_e g\left( \frac{\sigma}{h_e}v - \nabla v \cdot n\right)~ds.
\end{equation}
Now we can state the DG scheme to be addressed in this paper: 
\begin{equation}
\textup{(DG;$f,g$)}\quad 
\text{Find}\quad u_h \in V_h \quad \text{ s.t. }\quad  
a(u_h,\chi) = F(\chi) \quad {}^\forall \chi \in V_h.
\label{eq:dg}
\end{equation}
This scheme is usually called the symmetric interior penalty DG (SIPDG) method, and the $L^2$ theory is well-developed at present (see \cite{MR1885715}). For example, the DG bilinear form $a$ is continuous in the sense that, for any $1 \le p \le \infty$, there exists $C>0$ satisfying
    \begin{equation}
    a(u,v) \le C \norm{u}_{V^p}\norm{v}_{V^{p'}} \quad {}^\forall u \in V^p, {}^\forall v \in V^{p'}. \label{eq:dgconti}
    \end{equation}
Moreover, there exists $\sigma_0>0$ such that, if $\sigma\ge \sigma_0$, then we have 
    \begin{equation}
    a(\chi,\chi) \ge C \norm{\chi}_{V^2}^2 \quad {}^\forall \chi \in V_h. \label{eq:dgcoercive}
    \end{equation}
Consequently, \textup{(DG;$f,g$)} with $\sigma\ge \sigma_0$ admits a unique solution $u_h \in V_h$ and, it satisfies 
\[
\|u_h\|_{V^2} \le \sup_{\chi\in V_h}\frac{F(\chi)}{\|\chi\|_{V^2}}.
\]
If the solution $u$ of \textup{(BVP;$f,g$)} belongs to $u \in H^s(\Omega)$ for some $s > \frac{3}{2}$, we have
\begin{equation}
    a(u,v) = F(v) \quad {}^\forall v \in V^2. \label{eq:dgconsistency}
\end{equation}
As a result, we have the Galerkin orthogonality (consistency) 
\begin{equation}
  a(u-u_h,\chi) = 0 \quad {}^\forall \chi \in V_h. \label{eq:galerkin}
\end{equation}
Our main theorem below will be formulated using the pair of functions $u\in V^\infty$ and $u_h\in V_h$ satisfying \eqref{eq:galerkin}. More generally, we consider $u\in V^\infty$ and $u_h\in V_h$ satisfying 
\begin{equation}
a(u-u_h,\chi) = 0 \quad {}^\forall \chi \in \mathring V_h. 
\label{eq:go}
\end{equation}
Below, we always assume that $\sigma\ge \sigma_0$. 

We are now in a position to state the main results of this paper, but to do so, we need additional notations. Suppose that we are given an open disk $D \Subset \Omega$ with center $x_0$ and radius $R$. 
In the disk $D$, we consider an auxiliary Neumann problem: 
\begin{equation}
\left\{
    \begin{array}{rcc}
    -\Delta w + w = \varphi & \text{in} & D \\
\partial_n w = 0 & \text{on} & \partial D,
\end{array}\right. 
\label{eq:circle}
\end{equation}
where $\partial_n=n\cdot \nabla$ denotes the outward normal derivative to $\partial D$. 
Because $\partial D$ is smooth, for a given $\varphi\in L^2(D)$, there exists a unique solution $w\in H^2(D)$ of \eqref{eq:circle}; this correspondence is denoted by $w=\mathcal{G}_D\varphi$. We recall the $W^{2,p}$ and $W^{1,1}$ regularity results in Section \ref{sect2}. The DG bilinear form $a_D^1$ corresponding to \eqref{eq:circle} is introduced as 
\begin{multline}
a_D^1(u,v)  \coloneqq \sum_{K \in \T_h} \int_{K \cap D}( \nabla u \cdot \nabla v  +  uv)~dx \\
-\sum_{e \in \E_h} \int_{e\cap \overline D}(\mean{\nabla u}\jump{v} + \mean{\nabla v}\jump{u})~ds 
 + \sum_{e \in \E_h}\frac{\sigma}{h_e}\int_{e\cap \overline D} \jump{u}\jump{v} ~ds. \label{eq:bilinearoncircle}
\end{multline}

We introduce the operator $\Pi^1_h$ of $L^2(D)\to V_h(D)$ as
\begin{equation}
a_D^1(\mathcal{G}_D\varphi - \Pi_h^1 \varphi , \chi) = 0 \quad {}^\forall \chi \in V_h(D) \label{eq:circlegalerkin}
\end{equation}
and make the assumption below: 
\begin{asm}\label{asm1}
There exist a function $\alpha$ of $\R_+=(0,\infty)\to \R_+$ and constant $C>0$ which are independent of $h$ such that 
\begin{gather}
\mbox{$\alpha$ is bounded in a neighborhood of $0$;}\label{eq:asm1a}\\
a_D^1(\mathcal{G}_D^1 \varphi -\Pi_h^1 \varphi,v) \le C h\alpha(h) \norm{\varphi}_{L^2(D)}\norm{v}_{V^{\infty}(D)} \quad {}^\forall \varphi \in L^2(D). {}^\forall v \in V^{\infty}(D) \label{eq:asm1}
\end{gather}
for a sufficiently small $h$.
\end{asm}

\begin{rem}
In view of \eqref{eq:circleconti} and \eqref{eq:circlel1} of Proposition \ref{prop:D}, 
we can take at least $\alpha(h) = 1$.
\end{rem}

Using this $\alpha(h)$, we define $\norm{\cdot}_{\alpha(h),\Omega_0}$ as
\begin{equation}
\norm{v}_{\alpha(h),\mathcal{O}} \coloneqq \norm{v}_{L^\infty(\mathcal{O})} + \alpha(h)\norm{v}_{V^\infty(\mathcal{O})}
\end{equation}
for $\mathcal{O}\subset \Omega$ or $\mathcal{O}=\Omega$. 

Our first result is the following interior error estimate in the $L^\infty$ norm. 


\begin{thm}[$L^\infty$ interior error estimate]\label{thm1}
Letting $u \in V^\infty,u_h \in V_h$ satisfy \eqref{eq:go} and supposing that $\kappa>0$ and open sets $\Omega_0 \subset \Omega_1 \subset \Omega$ satisfy $\operatorname{dist}(\Omega_0,\partial\Omega_1) \ge \kappa h$, then we have, under Assumption \ref{asm1}, 
\begin{equation}
\norm{u-u_h}_{L^\infty(\Omega_0)} \le C \left(\inf_{\chi \in V_h}\norm{u-\chi}_{\alpha(h),\Omega_1}+\norm{u-u_h}_{L^2(\Omega_1)}\right) \label{eq:localerror}
\end{equation}
for a sufficiently small $h$.
\end{thm}

The next result is the weak discrete maximum principle.

\begin{thm}[Weak discrete maximum principle]\label{thm2}
Supposing that Assumption \ref{asm1} is satisfied and letting $u_h \in V_h$ be the discrete harmonic function, i.e.,  
    \begin{equation}
    a(u_h,\chi) = 0 \quad{}^\forall \chi \in \mathring V_h, \label{eq:discreteharmonic}
    \end{equation}
then we have 
    \begin{equation}
    \norm{u_h}_{L^\infty(\Omega)} \le C \norm{u_h}_{ L^\infty(\partial \Omega)} \label{eq:dmp}
    \end{equation}
    for a sufficiently small $h$. 
    \end{thm}

To state the final $L^\infty$ error estimate, we make the following assumption on triangulations.

\begin{asm}\label{asm2}
There exist a convex polygonal domain $\widetilde \Omega \Supset \Omega$ and its triangulation $\widetilde \T_h$
such that $\T_h$ is the restriction of $\widetilde \T_h$ to $\Omega$, and that 
\eqref{eq:triangulation} holds for any $\widetilde K \in \widetilde \T_h\in \{\widetilde \T_h\}_h$ with the same constant $C_*$.
\end{asm}

We define $\widetilde \E_h$, $V^p(\widetilde \Omega)$, $V_h(\widetilde \Omega)$ similarly, using $\widetilde \T_h$.

\begin{thm}[$L^\infty$ error estimate]\label{thm3}
Letting $u \in V^\infty,u_h \in V_h$ satisfy \eqref{eq:go}, and supposing that Assumptions \ref{asm1} and \ref{asm2} are satisfied, then we have 
\begin{equation}
\norm{u-u_h}_{L^\infty(\Omega)} \le C \left(\inf_{\chi \in V_h}\norm{u-\chi}_{\alpha(h),\Omega} + \norm{u-u_h}_{L^\infty(\partial \Omega)}\right) \label{eq:thm3}
\end{equation}
for a sufficiently small $h$.
\end{thm}

\section{Preliminaries}
\label{sect2}

In this section, we collect some preliminary results. 

\subsection{Some local estimates}
\label{sec:l2}

For $x_0 \in \R^2$ and $d>0$, we denote $B_d(x_0)$ as an open disk with center $x_0$ and radius $d$. We define $N_d(\Omega_0) \coloneqq \{ x \in \overline\Omega \colon \dist{(x,\Omega_0)} < d \}$ and $S_d(x_0) \coloneqq N_d(\{x_0\}) = \overline\Omega \cap B_d(x_0)$.
For $\Omega_0 \subset \Omega_1 \subset \Omega$, we define
\[
    d(\Omega_0,\Omega_1) \coloneqq \dist{(\Omega_0,\partial \Omega_1)},\quad  d_{\Omega}(\Omega_0,\Omega_1) \coloneqq \dist{(\Omega_0,\partial \Omega_1 \setminus \partial\Omega}).
\]
If $d_{\Omega}(\Omega_0,\Omega_1) \ge d$, then we have $N_d(\Omega_0) \subset \Omega_1$.

First, we recall further regularity results for the solution of \textup{(BVP;$f,0$)}; see \cite[Lemma 1.2, Lemma 1.3]{MR551291} for more detail.  

\begin{prop}
\label{prop:poisson2}
Under the same settings of Proposition \ref{prop:poisson}, we have the following. 
    \begin{enumerate}
    \item[\textup{(iii)}] \label{prop:poisson3} Assume that $f \in L^2(\Omega)$ and $\supp{f} \subset S_d(x_0)$ for some $d > 0$ and $x_0 \in \overline \Omega$ with $\dist(x_0,\partial \Omega) \le d$.
    Then, we have
    \begin{equation}
    \abs{u}_{H^1(\Omega)} \le C d \norm{f}_{L^2(S_d(x_0))}. \label{eq:localstability}
    \end{equation}
    \item[\textup{(iv)}] \label{prop:poisson4} Assume that $\Omega_0 \subset \Omega_1 \subset \Omega$ and $d>0$ with $d_{\Omega}(\Omega_0,\Omega_1) \ge d$.
    Then, there exists a positive constant $C$ such that 
    \begin{equation}
    \abs{u}_{W^{2,p}(\Omega_0)} \le C(\abs{f}_{L^p(\Omega_1)}+d^{-1}\abs{u}_{W^{1,p}(\Omega_1)}+d^{-2}\norm{u}_{L^p(\Omega_1)}) \label{interiorstability}
    \end{equation}
    under the same assumption of \textup{(i)} or \textup{(ii)} in Proposition \ref{prop:poisson}.
    \end{enumerate}
\end{prop}

A version of the Poincar\'e inequality is available (see \cite[Lemma 1.1]{MR551291}).

\begin{prop}\label{prop:poincare}
    Let $\Omega$ be a simply connected polygonal domain. Then, there exists a positive constant $C$ satisfying
    \begin{equation}
    \norm{v}_{L^2(S_d(x_0))} \le C d\abs{v}_{H^1(S_d(x_0))}  \quad {}^\forall v \in H^1_0(\Omega)\label{eq:poincare} 
    \end{equation}
    for all $x_0 \in \partial \Omega$ and $d>0$.
\end{prop}

\begin{prop}\label{prop:holder}
For a domain $\Omega_0 \subset \Omega$ and $1 \le p < q \le \infty$, we have $V^q(\Omega_0) \subset V^p(\Omega_0)$.
In particular, there exists a positive constant $C$ independent of $h$ and $\Omega_0$ satisfying
\begin{equation}
\norm{v}_{V^p(\Omega_0)} \le C \abs{\Omega_0}_2^{\frac{1}{p}-\frac{1}{q}} \norm{v}_{V^q(\Omega_0)} \label{eq:holder}
\end{equation}
for a sufficiently small $h$.
\end{prop}

\begin{proof}
Using H\"older's inequality, we have
\begin{align*}
    \norm{v}_{V^p(\Omega_0)}^p &\le C \abs{\Omega_0}_2^{\frac{q-p}{q}}\sum_{K \in \T_h} \norm{v}_{W^{1,q}(K \cap \Omega_0)}^p
    \\ &+ \left(\sum_{e \in \E_h} h_e \abs{e\cap\overline\Omega_0}_1\right)^{\frac{q-p}{q}}\left(\sum_{e \in \E_h} h_e^{1-q}\norm{\jump{v}}_{L^q(e \cap \overline \Omega_0 )}^p
 + \sum_{e \in \E_h} h_e \norm{\mean{\nabla v}}_{L^q(e \cap \overline \Omega_0)}^p\right).
\end{align*}
Therefore, 
\begin{align*}
    \sum_{e \in \E_h} h_e \abs{e\cap\overline\Omega_0}_1 \le \sum_{K \in \T_h(\Omega_0)}\abs{K}_2 \le C\abs{\Omega_0}_2, 
\end{align*}
where $\T_h(\Omega_0) \coloneqq \{K \in \T_h \colon \overline K \cap \overline \Omega_0 \ne \emptyset\}$. Consequently, \eqref{eq:holder} follows. 
\end{proof}
    For $\mathcal{O}\subset\Omega$, we denote broken Sobolev space $W^{j,p}_h(\mathcal{O})$ as 
    \[
        W^{j,p}_h(\mathcal{O})=W^{j,p}_h(\mathcal{O},\T_h) \coloneqq\{v \in L^p(\mathcal{O}) \colon v|_{K\cap \mathcal{O}} \in W^{j,p}(K\cap \mathcal{O})\}
    \]
    equipped with the norm
    \[
        \norm{v}_{W^{j,p}_h(\mathcal{O})} = \left(\sum_{K \in \T_h} \norm{v}_{W^{j,p}( K \cap \mathcal{O})}^p\right)^{{1}/{p}}.
    \]
The following results are available; see \cite[Chapter 3]{MR0520174}, \cite[Propositions 2.1 and 2.2]{MR2113680} and \cite[Proposition 2.2]{MR0431753}.

\begin{prop}\label{prop:bestapprx}
    Let $1 \le p \le \infty$, and $0 \le i \le 1 \le j \le 1 + r$.
    Assume that $\kappa>0$ and open sets $\Omega_0 \subset \Omega_1 \subset \Omega$ satisfy $d_{\Omega}(\Omega_0,\Omega_1)\ge\kappa h$.
    Then, there exists positive constant $C$ independent of $h$ such that for $v \in W^{j,p}_h(\Omega_1)$, $\chi \in V_h(\Omega_1)$ exists and satisfies
    \begin{equation}
    \norm{v-\chi}_{W^{i,p}_h(\Omega_0)} \le C h^{j-i}\norm{v}_{W^{j,p}_h(\Omega_1)}. \label{eq:bestapprx}
    \end{equation}
\end{prop}

\begin{prop}\label{prop:inverseineq}
    Let $1 \le p \le \infty$,
    Assume that $\kappa>0$ and open sets $\Omega_0 \subset \Omega_1 \subset \Omega$ satisfy $d_{\Omega}(\Omega_0,\Omega_1)\ge\kappa h$.
    Then, there exists a positive constant $C$ independent of $h$ satisfying
    \begin{equation}
    \norm{v_h}_{V^p(\Omega_0)} \le C h^{-1} \norm{v_h}_{L^p(\Omega_1)} \label{eq:inverseineq}
    \end{equation}
    for $v_h \in V_h(\Omega_1)$.
\end{prop}
    
 \begin{prop}\label{prop:superapprx}
    Let open sets $\Omega_1 \Subset \Omega_2 \Subset \Omega_3 \Subset\Omega_4 \Subset \Omega$.
    Then, there exists a positive constant $C$ independent of $h$, and the following property holds for a sufficiently small $h$.
    For each $\eta_h \in V_h(\Omega_4)$, there exists $\chi \in \mathring V_h(\Omega_3)$ satisfying $\chi \equiv \eta_h$ on $\Omega_2$ and
    \begin{equation}
    \norm{\eta_h-\chi}_{V^2(\Omega_3)} \le C \norm{\eta_h}_{V^2(\Omega_4\setminus \Omega_1)}.
    \end{equation}
\end{prop}

\subsection{$L^2$ theory for DG method}
\label{sec:l2dg}

The following results, Propositions \ref{prop:energyerror} and \ref{prop:interiorl2}, are well-known (see \cite[\S 4]{MR1885715} and \cite[\S 3 and \S 4]{MR2113680}). 

\begin{prop}\label{prop:energyerror}
For $f\in L^2(\Omega)$ and $g\in H^{1/2}(\Gamma)$, 
    there exists a unique solution $u_h \in V_h$ of DG scheme \textup{(DG;$f,g$)}. 
    In addition, if the solution $u$ of \eqref{eq:poisson} belongs to $H^s(\Omega)$ with $s>\frac{3}{2}$, then, we have
    \begin{equation}
    \norm{u-u_h}_{V_2(\Omega)} \le C \inf_{\chi \in V_h}\norm{u-\chi}_{V_2(\Omega)}. \label{eq:energyestimate} 
    \end{equation}
Moreover, if $u \in H^{r+1}(\Omega)$, we have 
    \begin{equation}
\|u-u_h\|_{L^2(\Omega)}+h\norm{u-u_h}_{V_2(\Omega)} \le C h^{r+1}|u|_{H^2(\Omega)}.
\label{eq:energyestimate1} 
    \end{equation}
\end{prop}

\begin{prop}\label{prop:interiorl2}
    Assume that $\kappa>0$ and open sets $\Omega_0 \subset \Omega_1 \subset \Omega$ satisfy $d = d_{\Omega}(\Omega_0,\Omega_1) \ge \kappa h$.
    We set $S_h = V_h$ or $\mathring V_h$.
    If $S_h = V_h$, we also assume that $d(\Omega_0,\Omega_1) >0$.
    If $u \in H^1(\Omega)$ and $u_h \in S_h$ satisfy
    \[
        a(u-u_h,\chi) = 0 \quad {}^\forall \chi \in \mathring V_h
    \]
    Then, there exists a positive constant $C_1$ independent of $h$, $u$, and $u_h$ satisfying
    \begin{equation}
    \norm{u-u_h}_{V^2(\Omega_0)} \le C_1\left( \inf_{\chi \in S_h} \norm{u-\chi}_{V^2(\Omega_1)} + \norm{u-u_h}_{L^2(\Omega_1)}\right). \label{dg:interiorerror1}
    \end{equation}
    Moreover, if $u \in H^{1+r}(\Omega)$, there exists a positive constant $C_2$ independent of $h$, $u$, $u_h$, and $d$ satisfying
    \begin{equation}
    \norm{u-u_h}_{V^2(\Omega_0)} \le C_2\left(h^r \norm{u}_{H^{1+r}(\Omega_1)}+d^{-1}\norm{u-u_h}_{L^2(\Omega_1)}\right). \label{eq:interiorerror2}
    \end{equation}
\end{prop}


\subsection{Estimates on the disk}
\label{sec:disk}

We here state some estimates for functions defined on the open disk $D$ with center $x_0$ and radius $R$. Recall that we consider the Neumann boundary value problem \eqref{eq:circle} and the corresponding bilinear form $a_D^1$ defined as \eqref{eq:bilinearoncircle}. 
For the positive constant $c$, we denote by $cD$ an open disk with center $x_0$ and radius $cR$. 

The following property (i) is well-known (see \cite{MR0188615} for example). However, we can find no explicit reference to (ii), and we prove it by essentially the same way as \cite{MR0336050}.

\begin{prop}\label{prop:circleregulality}
    \begin{enumerate}
    \item[\textup{(i)}] If $f \in L^p(D)$ for some $1 < p < \infty$, then the solution $u \in W^{2,p}(D)$ of \eqref{eq:circle} exists and satisfies
    \begin{equation}
    \norm{u}_{W^{2,p}(D)} \le C \norm{f}_{L^p(D)}. \label{eq:circleregulality}
    \end{equation}
    \item[\textup{(ii)}] If $f \in L^1(D)$, then the weak solution $u \in W^{1,1}(D)$ of \eqref{eq:circle} exists and satisfies
    \begin{equation}
    \norm{u}_{W^{1,1}(D)} \le C \norm{f}_{L^1(D)}. \label{eq:circlel1}
    \end{equation}
    \end{enumerate}
\end{prop}

\begin{prop}
\label{prop:D}
\textup{(i)} Consistency. 
   If the solution $u$ of \eqref{eq:circle} belongs to $u \in H^s(D)$ with $s > \frac{3}{2}$, then we have
    \begin{equation}
    a_D^1(u,v) = (f,v)_D \quad {}^\forall v \in V^2(D). \label{eq:circleconsistency}
    \end{equation}
\textup{(ii)} Continuity. For $1 \le p \le \infty$, we have
    \begin{equation}
    a_D^1(u,v) \le C \norm{u}_{V^p(D)}\norm{v}_{V^{p'}(D)} \quad {}^\forall u \in V^p(D), {}^\forall v \in V^{p'}(D). \label{eq:circleconti}
    \end{equation}
\textup{(iii)} Coercivity. 
   \begin{equation}
    a_D^1(\chi,\chi) \ge C \norm{\chi}_{V^2(D)}^2 \quad {}^\forall \chi \in V_h(D). \label{eq:circlecoercive}
    \end{equation}
\end{prop}

\begin{lem}\label{lem1}
Assume that $\tilde u \in V^\infty(D)$ satisfies $\supp \tilde u \subset \frac{1}{2}D$.
Let $\tilde u_h \in V_h(D)$ satisfy
\[
a_D^1(\tilde u - \tilde u_h ,\chi) = 0 \quad {}^\forall \chi \in V_h(D).
\]
Then, under Assumption \ref{asm1}, there exists a positive constant $C$ independent of $h$ and $\tilde u$ satisfying
\begin{equation}
\norm{\tilde u - \tilde u_h}_{L^\infty(\frac{1}{4}D)} \le C \norm{\tilde u}_{\alpha(h),D} \label{eq:lem1}
\end{equation}
for a sufficiently small $h$.
\end{lem}
\begin{proof}

We take $x_1 \in \frac{1}{4}D$ such that $\abs{\tilde u(x_1)-\tilde u_h(x_1)} = \norm{\tilde{u}-\tilde{u}_h}_{L^\infty(\frac{1}{4}D)}$.
For a sufficiently large $M$, we denote by $D_h \subset D$ the open disk with center at $x_1$ and radius $Mh$. Then, we have
\begin{align*}
\abs{\tilde u(x_1)-\tilde u_h (x_1)} & \le \abs{\tilde u(x_1)} + \abs{\tilde u_h(x_1)} \\
 & \le \abs{\tilde u (x_1)} + Ch^{-1}\abs{\tilde u_h}_{L^2(D_h)} \\
 & \le C\norm{\tilde u}_{L^\infty(D_h)} + Ch^{-1}\norm{\tilde u - \tilde u_h}_{L^2(D_h)}.
\end{align*}
For $\phi \in C_0^\infty(D_h)$, we set $v = \mathcal{G}_D \phi$ and $v_h = \Pi_h^1 \phi$.
Then, 
\begin{align*}
    (\tilde u - \tilde u_h,\phi)_{D_h} &= a_D^1(\tilde u -\tilde u_h , v) = a_D^1(\tilde u - \tilde u_h,v-v_h) =a_D^1(\tilde u ,v-v_h)\\
    & \le C h \alpha(h)\norm{\tilde u}_{V^\infty(D)}\norm{\phi}_{L^2(D_h)}.
\end{align*}
Therefore, we deduce $\norm{\tilde u -\tilde u_h}_{L^2(D_h)} \le Ch\alpha(h) \norm{\tilde u}_{V^\infty(D)}$, which implies \eqref{eq:lem1}.
\end{proof}

\begin{lem}\label{lem2}
    Assume that $w_h \in V_h(D)$ satisfies
    \[
        a_D^1(w_h,\chi) = 0 \quad {}^\forall \chi \in \mathring V_h(D).
    \]
    Then, there exists a positive constant $C$ independent of $h$ and $w_h$ satisfying
    \begin{equation}
    \abs{w_h(x_0)} \le C \norm{w_h}_{L^2(D)} \label{eq:lem2}
    \end{equation}
    for a sufficiently small $h$.
\end{lem}
\begin{proof}
    In view of Proposition \ref{prop:superapprx}, there exists $\eta \in \mathring V_h(\frac{3}{4}D)$ satisfying $\eta_h \equiv w_h$ on $\frac{1}{2}D$ and
    $\norm{\eta_h}_{V^2(\frac{3}{4}D)} \le C \norm{w_h}_{V^2(D)}$.
    For a sufficiently large $M$, let $D_h \subset D$ be the disk with center at $x_0$ and radius $Mh$.
    Then, 
    \begin{align}
    \abs{w_h(x_0)} = \abs{\eta_h(x_0)} \le C h^{-1} \norm{\eta_h}_{L^2(D_h)}. \label{eq:lem2first}
    \end{align}

    For $\phi \in C_0^\infty(D_h)$, we set $v = \mathcal{G}_D \phi$ and $v_h = \Pi_h^1 \phi$.
    According to Proposition \ref{prop:superapprx}, there exists $\chi_h \in \mathring V_h(\frac{1}{2}D)$ satisfying $\chi_h \equiv v_h$ on $\frac{1}{4}D$ and
    $\norm{v_h-\chi_h}_{V^2(\frac{3}{4}D)} \le C \norm{v_h}_{V^2(\frac{3}{4}D\setminus\frac{1}{8}D)}$.
    We can estimate this as
    \begin{align}
    (\eta_h,\phi)_{D_h} &= a_D^1(\eta_h,v) = a_D^1(\eta_h,v_h) = a_D^1(\eta_h,v_h-\chi_h) \nonumber \\
     & \le C \norm{\eta_h}_{V^2(\frac{3}{4}D)} \norm{v_h-\chi_h}_{V^2(\frac{3}{4}D)} \nonumber \\
     & \le C \norm{w_h}_{V^2(D)}\norm{v_h}_{V^2(\frac{3}{4}D\setminus\frac{1}{8}D)}.
    \end{align}
Because $a_D^1(v_h,\chi) = (\phi,\chi)_D = 0$ for $\chi \in \mathring V_h(D\setminus D_h)$, using Proposition \ref{prop:interiorl2} and the Sobolev inequality, we obtain
    \begin{align}
    \norm{v_h}_{V^2(\frac{3}{4}D \setminus \frac{1}{8}D)} &\le C \norm{v_h}_{L^2(D)} \nonumber\\
     & \le C\left(\norm{v-v_h}_{W^{1,1}_h(D)} + \norm{v}_{W^{1,1}(D)}\right),
    \end{align}
    where $\norm{v}_{W^{1,1}_h(D)} \coloneqq \sum_{K \in \T_h} \norm{v}_{W^{1,1}(K \cap D)}$.
    By Propositions \ref{prop:energyerror} and \eqref{eq:circleregulality}, we have
    \begin{align}
    \norm{v-v_h}_{W^{1,1}_h(D)} &\le C\left(\sum_{K \in \T_h}\norm{v-v_h}_{H^1(K\cap D)}^2\right)^{1/2} \nonumber \\
    & \le C h \norm{v}_{H^2(D)} \nonumber \\
    & \le C h \norm{\phi}_{L^2(D_h)}.
    \end{align}
    Using Proposition \ref{prop:circleregulality}, we have
    \begin{equation}
        \norm{v}_{W^{1,1}(D)} \le C \norm{\phi}_{L^1(D_h)} \le Ch\norm{\phi}_{L^2(D_h)}. \label{eq:lem2last}
    \end{equation}
    From \eqref{eq:lem2first}--\eqref{eq:lem2last}, we obtain \eqref{eq:lem2}.
\end{proof}

\section{Interior error estimates (Proof of Theorem \ref{thm1})}
\label{s:I}

We first consider the homogeneous Neumann boundary value problem in $\Omega$.
Set $a^1(u,v) \coloneqq a(u,v) + (u,v)_\Omega$.

\begin{lem}\label{thm1:coercive}
    Assume that $\kappa>0$ and open sets $\Omega_0 \subset \Omega_1 \subset \Omega$ satisfy $d = d(\Omega_0,\Omega_1) \ge \kappa h$.
    Let $u \in V^\infty$ and $u_h \in V_h$ satisfy
\[
a^1(u-u_h,\chi) = 0 \quad {}^\forall \chi \in \mathring V_h.
\]
Then, under Assumption \ref{asm1}, there exists a positive constant $C$ independent of $h$, $u$, and $u_h$ satisfying
\begin{equation}
\norm{u-u_h}_{L^\infty(\Omega_0)} \le C \left(\inf_{\chi \in V_h}\norm{u-\chi}_{\alpha(h),\Omega_1}+\norm{u-u_h}_{L^2(\Omega_1)}\right) \label{eq:coercivelocalerror}
\end{equation}
for a sufficiently small $h$.
\end{lem}

\begin{proof}
Letting $\chi\in V_h$ be arbitrary, we set $v=u-\chi$ and $v_h=u_h-\chi$.  
We take $x_0 \in \Omega_0$ such that $\abs{v(x_0)-v_h(x_0)}= \norm{v-v_h}_{L^\infty(\Omega_0)}= \norm{u-u_h}_{L^\infty(\Omega_0)}$.  
Let $D\Subset \Omega_1$ be an open disk with center at $x_0$ and radius $R < d$.
Because $d \ge \kappa h$, we can take $R$ independent of $x_0$.
Letting $\omega \in C^\infty_0(\frac{1}{2}D)$ satisfy $0 \le \omega \le 1 $ and $\omega \equiv 1$ on $\frac{1}{4}D$, we set $\tilde v = \omega v \in V^\infty(D)$. We define $\tilde v_h \in V_h(D)$ as
\[
a_D^1(\tilde v - \tilde v_h ,\xi) = 0 \quad {}^\forall \xi \in V_h(D).
\]
Then, $a_D^1(\tilde v_h - v_h,\eta_h) = a^1(v-v_h,\eta_h) = 0$ for $\eta_h \in \mathring V_h(\frac{1}{4}D)$, because $v=v_h$ in $\supp \eta_h\subset \frac14 D$.
Using Lemmas \ref{lem2} and \ref{lem1}, we have
\begin{align*}
\abs{\tilde v_h(x_0) - v_h(x_0)} 
&\le C \norm{\tilde v_h -v_h}_{L^2(\frac{1}{4}D)} \\
& \le C \norm{\tilde v -\tilde v_h}_{L^\infty(\frac{1}{4}D)} + C \norm{v-v_h}_{L^2(\frac{1}{4}D)} \\
& \le C \norm{\tilde v}_{\alpha(h),\frac{1}{2}D} + C \norm{v-v_h}_{L^2(\frac14 D)}\\
& \le C \norm{v}_{\alpha(h),D} + C \norm{v-v_h}_{L^2(\frac14 D)}.
\end{align*}
Therefore, 
\begin{align*}
 \norm{u-u_h}_{L^\infty(\Omega_0)}&= \abs{v(x_0)-v_h(x_0)} \\
 & \le \abs{\tilde v(x_0)-\tilde v_h(x_0)} + \abs{\tilde v_h(x_0) - v_h(x_0)}\\
 & \le C \norm{v}_{\alpha(h),D} + C \norm{v-v_h}_{L^2(\frac14 D)}\\
&\le C\norm{u-\chi}_{\alpha(h),\Omega_1} + C\norm{u-u_h}_{L^2(\Omega_1)}.
\end{align*}
\end{proof}

\begin{lem}\label{lem3}
    Assume that $D' \Subset D \subset \Omega$ are open disks with the same center.
    Let $u \in V^\infty(D)$ and $u_h \in V_h(D)$ satisfy
    \[
        a_D(u-u_h,\chi)=0\quad{}^\forall \chi \in \mathring V_h(D)
    \]
    where $a_D(v,w) = a_D^1(v,w)-(v,w)_D$.
    Then, letting Assumption \ref{asm1} be satisfied, we have, for $p$ and $q$ which satisfy $2 \le q < p \le \infty$ and $\frac{1}{q} < \frac{1}{p} + \frac{1}{2} \displaystyle$
    \begin{equation}
    \norm{u-u_h}_{L^p(D')} \le C \left(\norm{u}_{\alpha(h),D} + \norm{u-u_h}_{L^q(D)}\right) \label{eq:lem3}
    \end{equation}
    for a sufficiently small $h$.
    \end{lem}
    
    \begin{proof}
    Setting $\psi = \mathcal{G}_D(u-u_h)$, $\psi_h = \Pi_h^1(u-u_h)$, we have 
    \[
        a_D^1(u-u_h-\psi_h,\chi) = 0 \quad{}^\forall \chi \in \mathring V_h(D).
    \]
    We apply Lemma \ref{thm1:coercive} to obtain
    \begin{equation*}
    \norm{u-u_h-\psi_h}_{L^\infty(D')} \le C\norm{u}_{\alpha(h),D} + C\norm{u-u_h}_{L^2(D)} + C\norm{\psi_h}_{L^2(D)}.
    \end{equation*}
On the other hand, using \eqref{eq:circlecoercive}, we have
    \begin{equation*}
    \norm{\psi_h}_{L^2(D)}^2 \le C a_D(\psi_h,\psi_h)= C (u-u_h,\psi_h)_{D} \le C\norm{u-u_h}_{L^2(D)}\norm{\psi_h}_{L^2(D)}
    \end{equation*}
and, therefore, 
    \begin{align}
    \norm{u-u_h}_{L^p(D')}& \le \norm{u-u_h-\psi_h}_{L^p(D')} + \norm{\psi_h}_{L^p(D')} \nonumber \\
     & \le C \norm{u}_{\alpha(h),D} + C\norm{u-u_h}_{L^2(D)} + \norm{\psi_h}_{L^p(D')}. \label{thm1coercive:second}
    \end{align}
    Because $\psi$ is smooth, we again apply Lemma \ref{thm1:coercive} to obtain
    \begin{align*}
    \norm{\psi_h}_{L^\infty(D')} & \le \norm{\psi-\psi_h}_{L^\infty(D')} + \norm{\psi}_{L^\infty(D')} \\
    & \le \norm{\psi}_{\alpha(h),D} + C\norm{\psi-\psi_h}_{L^2(D)} \\
    &\le C \norm{\psi}_{\alpha(h),D} \\
    &\le C\left(\norm{\psi}_{W^{1,\infty}(D)} + \max_{e \in \E_h,e\cap \overline D \ne \emptyset}\norm{\nabla \psi}_{L^\infty(e\cap\overline D)}\right).
    \end{align*}
    The Sobolev inequality and elliptic regularity give
    \[
        \norm{\psi}_{W^{1,\infty}(D)} \le C \norm{\psi}_{W^{2,s}(D)} \le C \norm{u-u_h}_{L^s(D)}
    \]
    for $s >2$.
    Because
    \[
        \abs{\nabla \psi (x)} \le C \int_D \frac{\abs{u(y)-u_h(y)}}{\abs{x-y}} dy \le C \norm{u-u_h}_{L^s(D)}
    \]
    for $x \in D$, we have
    \begin{equation}
    \norm{\psi_h}_{L^\infty(D')} \le C \norm{u-u_h}_{L^s(D)}. \label{thm1:linfty}
    \end{equation}
    Similarly, we deduce  
    \begin{align*}
    \norm{\psi_h}_{L^2(D')} &\le C \norm{\psi}_{V^2(D)} \\
    & \le C\norm{\psi}_{H^1(D)} + C \left( \sum_{e \in \E_h,e\cap \overline D \ne \emptyset} h_e\norm{\nabla \psi}_{L^2(e\cap\overline D)}^2\right)^{1/2}.
    \end{align*}
    Applying the Young inequality for convolution, we have
    \begin{align*}
        \abs{\nabla \psi(x)} &\le C \int_D\frac{\abs{u(y)-u_h(y)}}{\abs{x-y}} dy \\
        &\le C \norm{u-u_h}_{L^t(D)} \norm{\abs{x-y}^{-1}}_{L^{2t/(3t-2)}(D)} \\
        & \le C \norm{u-u_h}_{L^t(D)} .
    \end{align*}
for $x \in D$ and $1<t<2$. 
    Therefore, 
    \begin{align}
    \norm{\psi_h}_{L^2(D')} &\le C\norm{\psi}_{W^{2,t}(D)} + C \left( \sum_{e \in \E_h,e\cap \overline D \ne \emptyset}h_e^2\norm{u-u_h}_{L^t(D)}^2\right)^{1/2} \nonumber \\
    & \le C \norm{u-u_h}_{L^t(D)}. \label{thm1:l2}
    \end{align}
    In view of the Riesz--Thorin interpolation theorem, we have
    \begin{align}
    \norm{\psi_h}_{L^p(D')} \le C \norm{u-u_h}_{L^q(D)}
    \end{align}
for $\theta = \frac{2}{p}$, and $\frac{1}{q} = \frac{1-\theta}{s} + \frac{\theta}{t} < \frac{1}{2}+\frac{1}{p}$.  This, together with \eqref{thm1coercive:second}, implies \eqref{eq:lem3}.  
    \end{proof}

We can now state the following proof. 

    \begin{proof}[Proof of Theorem \ref{thm1}.]
        Letting $\chi\in V_h$ be arbitrary, we set $v=u-\chi$ and $v_h=u_h-\chi$.
        Similar to the proof of Lemma \ref{thm1:coercive}, we take $x_0 \in \Omega_0$, $D \Subset \Omega_1$, and $\omega \in C^\infty_0(\frac{1}{2}D)$.
        Setting $\tilde v = \omega v \in V_h(D)$, we define $\tilde v_h \in V_h(D)$ as
        \[
        a_D^1(\tilde v - \tilde v_h ,\xi) = 0 \quad {}^\forall \xi \in V_h(D).
        \]
        Then, Lemma \ref{lem1} gives
        \[
            \norm{\tilde v -\tilde v_h}_{L^\infty(\frac{1}{4}D)} \le C\norm{\tilde v}_{\alpha(h),D}.
        \]
        Setting $D' = \varepsilon D$ for $\varepsilon < \frac{1}{4}$, we let $\psi$ be the unique solution of
        \[
            \left\{ \begin{array}{ccc}
            -\Delta \psi = -(\tilde v_h - \tilde v) & \text{in} & D' \\
            \psi = 0 & \text{on} & \partial D'
            \end{array} \right..
        \]
        Then, we have
        \[
            a_D(\psi,w) = -(\tilde v_h - \tilde v,w)_{D'} \quad {}^\forall w \in \{w \in V^2(D') \colon w|_{\partial D'} = 0 \}
        \]
        and 
        \[
            a_D(\psi-(v_h-\tilde v_h),\xi) = 0 \quad {}^\forall \xi \in \mathring V_h(D').
        \]
        Applying Lemma \ref{lem3} several times, we obtain
        \begin{align*}
        \norm{\psi-(v_h-\tilde v_h)}_{L^\infty(\frac{1}{4}D')} & \le C \norm{\psi}_{\alpha(h),D'} + C \norm{\psi-(v_h-\tilde v_h)}_{L^2(D')} \\
         & \le C \norm{\psi}_{\alpha(h),D'} + C \norm{v-v_h}_{L^2(D')} + C \norm{\tilde v - \tilde v_h}_{L^2(D')} \\
         & \le C \norm{\psi}_{\alpha(h),D'} + C \norm{v-v_h}_{L^2(D')} + C\norm{\tilde v}_{\alpha(h),D}.
        \end{align*}
        Then, we deduce $\norm{\psi}_{\alpha(h),D'} \le C \norm{\tilde v -\tilde v_h}_{L^\infty(D')} \le C\norm{v}_{\alpha(h),D}$ in the similar way as the proof of Lemma \ref{lem3}.  
        Using the triangle inequality, we have
        \begin{align*}
            \norm{v_h-\tilde v_h}_{L^\infty(\frac14D')}\le C\norm{v}_{\alpha(h),D} + C\norm{v-v_h}_{L^2(D)}.
        \end{align*}
        Therefore, 
        \begin{align*}
            \norm{u-u_h}_{L^\infty(\Omega_0)} &\le \abs{\tilde v(x_0)-\tilde v_h(x_0)} + \abs{\tilde v_h(x_0) - v_h(x_0)} \\
            & \le C\norm{v}_{\alpha(h),D} + C\norm{v-v_h}_{L^2(D)} \\
            &\le C\norm{u-\chi}_{\alpha(h),\Omega_1} + C\norm{u-u_h}_{L^2(\Omega_1)}.
        \end{align*}
 \end{proof}
        
\begin{cor}\label{cor1}
        Under the same assumption of Theorem \ref{thm1}, we further assume that $d(\Omega_1,\Omega) \ge \kappa h$ and $u \in W^{1+r,\infty}(\Omega)$.
        Then, there exists a positive constant $C$ independent of $h$, $u$, $u_h$, and $D$ satisfying
        \begin{equation}
        \norm{u-u_h}_{L^\infty(\Omega_0)} \le C h^r\left[h+\alpha\left(\frac{h}{d}\right)\right]\norm{u}_{W^{1+r,\infty}(\Omega_1)} + Cd^{-1}\norm{u-u_h}_{L^2(\Omega_1)} \label{eq:cor1}
        \end{equation}
        for a sufficiently small $h$.
\end{cor}

\section{Weak discrete maximum principle (Proof of Theorem \ref{thm2})}
\label{s:II}

We follow the same method as the proof of Theorem 1 of \cite{MR551291} to prove Theorem \ref{thm2} described below. 

\begin{proof}[Proof of Theorem \ref{thm2}.]
    Let $x_0 \in \Omega$ satisfy $\abs{u_h(x_0)} = \norm{u_h}_{L^\infty(\Omega)}$.
    Set $d = \dist(x,\partial \Omega)$. First, we consider the case $d \ge 2\kappa h$ for some $\kappa>1$. 
    Then, applying Corollary \ref{cor1} to an open disk with center $x_0$ and $u\equiv 0$, we have 
    \[
        \abs{u_h(x_0)} \le C d^{-1}\norm{u_h}_{L^2(S_{\frac{1}{2}d}(x_0))} \le C d^{-1}\norm{u_h}_{L^2(S_{d}(x_0))}.
    \]
    Now we assume that $d < 2\kappa h$. Using the inverse inequality, we have
    \[
        \abs{u_h(x_0)} \le Ch^{-1}\norm{u_h}_{L^2(S_{h}(x_0))}.
    \]
    Therefore,  
    \begin{equation}
        \norm{u_h}_{L^\infty(\Omega)} \le C \rho^{-1}\norm{u_h}_{L^2(S_\rho(x_0))} \label{eq:thm2step1}
    \end{equation}
    where $\rho = \max\{d,h\}$.

    Let $\phi \in C^\infty_0(S_\rho(x_0))$ satisfy $\norm{\phi}_{L^2(S_\rho(x_0))}=1$.
    Let $v \in H^1_0(\Omega)$ be the solution of \eqref{eq:poisson} with $f=\phi$ and $g=0$.
    Then, $v \in W^{2,p}(\Omega)$ for some $4/3< p\le 2 $, and $a(v,w) = (\phi,w)_\Omega$ for all $w \in V^2$. 
    Let $v_h \in \mathring V_h$ be satisfy
    \[
        a(v_h,\chi) = (\phi,\chi) \quad {}^\forall \chi \in \mathring V_h.
    \]
    In view of the assumption of $u_h$, we get
    \begin{align}
    \abs{(u_h,\phi)_\Omega} &= \abs{a(u_h,v)} = \abs{a(u_h,v-v_h)} \nonumber \\
    & = \abs{a(u_h-\chi,v-v_h)} \label{eq:thm2step2first}
    \end{align}
    for $\chi \in \mathring V_h$. 
    We define $\hat u_h \in \mathring V_h$ such that 
$\hat u_h = 0$ at nodal points on $\partial \Omega$ and $\hat u_h = u_h$ at interior nodal points. Then, we have  
    \begin{align*}
        \supp(u_h - \hat u_h) & \subset \Lambda_h = \{x \in \overline \Omega \colon \dist(x,\partial \Omega) \le h \},\\
    \norm{u_h - \hat u_h}_{L^\infty(\Omega)} & \le C \norm{u_h}_{L^\infty(\partial \Omega)}.
    \end{align*}
    Substituting \eqref{eq:thm2step2first} for $\chi = \hat u_h$, and using the inverse inequality, we have 
    \begin{align}
    \abs{(u_h,\phi)_\Omega} &\le C \norm{u_h-\hat u_h}_{V^\infty} \norm{v-v_h}_{V^1(\Lambda_h)} \nonumber \\
     & \le Ch^{-1}\norm{u_h - \hat u_h}_{L^\infty(\Omega)}\norm{v-v_h}_{V^1(\Lambda_h)} \nonumber \\
     & \le Ch^{-1}\norm{u_h}_{L^\infty(\partial \Omega)}\norm{v-v_h}_{V^1(\Lambda_h)}. \label{eq:thm2step2last} 
    \end{align}

    Set $R_0 = \diam\Omega$ and $d_j = R_02^{-j}$ for non-negative integer $j$.
    We define $A_j$ as 
    \[
        A_j \coloneqq \{x \in \overline \Omega \colon d_{j+1} \le \abs{x-x_0} \le d_j \}.
    \]
    Then, $\abs{A_j \cap \Lambda_h}_2 \le C d_jh$.
    Set $A_j^l =\displaystyle  \bigcup_{k=j-l}^{j+l}A_k$ and $J \coloneqq \min \{j \in {\mathbb Z} \colon d_{j+1} \le 8 \rho\}$.
    Then, we have
    \begin{align}
    \norm{v-v_h}_{V^1(\Lambda_h)} & \le \sum_{j=0}^J \norm{v-v_h}_{V^1(\Lambda_h \cap A_j)} + \norm{v-v_h}_{V^1(\Lambda_h \cap S_{8\rho})} \nonumber \\
     & \le C\sum_{j=0}^J h^{1/2}d_j^{1/2}\norm{v-v_h}_{V^2(\Lambda_h\cap A_j)} + C\rho^{1/2}h^{1/2} \norm{v-v_h}_{V^2(\Lambda_h \cap S_{8\rho}(x_0))}. \label{eq:thm2step3first}
    \end{align}
    To estimate the second term of \eqref{eq:thm2step3first}, we apply Propositions \ref{prop:energyerror} and \ref{prop:bestapprx} and get
    \begin{align}
    \norm{v-v_h}_{V^2(\Lambda_h \cap S_{8\rho}(x_0))} &\le Ch^{2-\frac{2}{p}}\norm{v}_{W^{2,p}(\Omega)} \nonumber \\
    & \le Ch^{2-\frac{2}{p}}\norm{\phi}_{L^p(S_\rho(x_0))} \nonumber \\
    & \le Ch^{2-\frac{2}{p}}\rho^{\frac{2}{p}-1} \label{eq:thm2step3circle}.
    \end{align}
    Meanwhile, using Proposition \ref{prop:interiorl2} for $j$ satisfying $\Lambda_h \cap A_j \ne \emptyset$, we have
    \begin{align*}
    \norm{v-v_h}_{V^2(\Lambda_h \cap A_j)} &\le C\left( \norm{v-v_h}_{V^2(A_j^1)} + d_j^{-1}\norm{v-v_h}_{L^2(A_j^1)}  \right) \nonumber \\
     & \le Ch^{2-\frac{2}{p}}\norm{v}_{W^{2,p}(A_j^2)} + C d_j^{-1} \norm{v-v_h}_{L^2(A_j^1)}.
    \end{align*}
    In view of Proposition \ref{prop:poisson4}, 
    \begin{align*}
    \norm{v}_{W^{2,p}(A_j^2)} &\le C\left( \norm{\phi}_{L^p(A_j^3)} + d_j^{-1}\abs{v}_{W^{1,p}(A_j^3)} + d_j^{-2}\norm{v}_{L^p(A_j^3)} \right) \\
    & \le C d_j^{\frac{2}{p}-1}\left(1 + d_j^{-1}\abs{v}_{H^1(A_j^4)} + d_j^{-2}\norm{v}_{L^2(A_j^4)} \right).
    \end{align*}
    Because $\diam A_j^4 \le 32 d_j$ and $\dist(A_j^4,\partial \Omega) \le h$, there exists $\overline x_j \in \partial \Omega$ satisfying
    \[
        A_j^4 \subset S_{64d_j}(\overline x_j)\,,\,S_\rho(x_0) \subset S_{64d_j}(\overline x_j).
    \]
    Therefore, we have
    \begin{align}
     \norm{v}_{W^{2,p}(A_j^2)} \le C d_j^{\frac{2}{p}-1} \label{eq:thm2step3w2p}
    \end{align}
    by Propositions \ref{prop:poisson3} and \ref{prop:poincare}.
    Let $\eta \in C^\infty_0(S_{64d_j}(\overline x_j))$ and let $w \in H^1_0(\Omega)$ be the solution of \eqref{eq:poisson} with $f=\eta$ and $g=0$.
    Let $w_h \in \mathring V_h$ satisfy
    \[
        a(w_h,\chi) = (\eta,\chi)_{\Omega} \quad {}^\forall \chi \in \mathring V_h.
    \]
    Similar to the above, we deduce 
    \begin{align*}
    \abs{(v-v_h,\eta)_{S_{64d_j}(\overline x_j)}} & = \abs{a(v-v_h,w-w_h)} \nonumber \\
     & \le C \norm{v-v_h}_{V^2}\norm{w-w_h}_{V^2} \nonumber \\
     & \le Ch^{4-\frac{4}{p}}\norm{\phi}_{L^p(S_\rho(x_0))}\norm{\eta}_{L^p(S_{64d_j}(\overline x_j))} \nonumber \\
     & \le Ch^{4-\frac{4}{p}}(\rho d_j)^{\frac{2}{p}-1}\norm{\eta}_{L^2(S_{64d_j}(\overline x_j))}.
    \end{align*}
    Therefore,
    \begin{align}
    \norm{v-v_h}_{L^2(A_j^1)} \le Ch^{4-\frac{4}{p}}(\rho d_j)^{\frac{2}{p}-1}. \label{eq:thm2step3l2}
    \end{align}
    Summing up \eqref{eq:thm2step3first}--\eqref{eq:thm2step3l2} and using $h \le \rho \le C d_j \le R_0$ and $p>\frac{3}{4}$, we obtain
    \begin{align}
    \norm{v-v_h}_{V^1(\Lambda_h)} & \le \sum_{j=0}^J h^{1/2}d_j^{1/2}\left(h^{2-\frac{2}{p}}d_j^{\frac{2}{p}-1} + h^{4-\frac{4}{p}}\rho^{\frac{2}{p}-1}d_j^{\frac{2}{p}-1} \right)
      + h\rho \left( \frac{h}{\rho} \right)^{\frac{3}{2}-\frac{2}{p}} \nonumber \\
    & \le Ch\rho. \label{eq;thm2step3last}
    \end{align}
    The desired \eqref{eq:dmp} now follows \eqref{eq:thm2step2last} and \eqref{eq:thm2step1}. 
\end{proof}

\section{$L^\infty$ error estimate (Proof of Theorem \ref{thm3})}
\label{s:III}

We finally state the following proof. 

\begin{proof}[Proof of Theorem \ref{thm3}.]
    Let $\tilde u \in V^\infty(\widetilde \Omega)$ be the extension of $u$ satisfying $\norm{\tilde u}_{\alpha(h), \widetilde \Omega} \le C \norm{u}_{\alpha(h),\Omega}$ and $\tilde u = 0$ on $\partial \widetilde \Omega$.
    Moreover, let $\tilde u_h \in \mathring V_h(\widetilde \Omega)$ solve 
    \[
        \tilde a(\tilde u - \tilde u_h ,\xi) = 0 \quad {}^\forall \xi \in \mathring V_h(\widetilde \Omega),
    \]
    where $\tilde a$ is the bilinear form \eqref{eq:bilinear} with replacement of $\T_h$, $\E_h$ by $\widetilde \T_h$, $\widetilde \E_h$, respectively.
    For arbitrary $\chi \in V_h$, we define $\tilde \chi \in \mathring V_h(\widetilde \Omega)$ as a zero extension.
    Then, in view of Theorem \ref{thm1}, we have
    \begin{equation}
    \norm{\tilde u -\tilde u_h}_{L^\infty(\Omega)} \le C \norm{\tilde u- \tilde \chi}_{\alpha(h),\widetilde \Omega} + C \norm{\tilde u - \tilde u_h}_{L^2(\widetilde \Omega)}. \label{eq:thm3step1}
    \end{equation}
    
    Let $\psi \in H^1(\widetilde \Omega)$ be the solution of
    \begin{equation}
    \left\{
    \begin{array}{ccc}
    -\Delta \psi = \tilde u - \tilde u_h & \text{in} & \widetilde \Omega \\
    \psi = 0 & \text{on} & \partial \widetilde \Omega .
    \end{array}\right.
    \end{equation}
    Then, $\psi \in H^2(\widetilde \Omega)$ and $\tilde a (\psi,\eta) = (\psi,\eta)_{\widetilde \Omega}$ for $\eta \in V^2(\widetilde \Omega)$.
    Let $\psi_h \in \mathring V_h(\widetilde \Omega)$ solve
    \[
        \tilde a(\psi-\psi_h,\xi) = 0 \quad {}^\forall \xi \in \mathring V_h(\widetilde \Omega).
    \]
    Then, by the continuity of $\tilde a$ and elliptic regularity, we have
    \begin{align}
    \norm{\tilde u -\tilde u_h}_{L^2(\widetilde \Omega)}^2 &= \tilde a(\tilde u-\tilde u_h ,\psi) = \tilde a(\tilde u-\tilde \chi,\psi-\psi_h) \nonumber \\
     & \le C\norm{\tilde u-\tilde \chi}_{V^\infty(\widetilde \Omega)}\norm{\psi - \psi_h}_{V^1(\widetilde \Omega)} \nonumber \\
     & \le Ch \norm{\tilde u-\tilde \chi}_{V^\infty(\widetilde \Omega)}\norm{\tilde u -\tilde u_h}_{L^2(\widetilde \Omega)}. \label{eq:thm3step2first}
    \end{align}
    Because $a(u_h - \tilde u_h,\xi) = 0$ for $\xi \in \mathring V_h$, using Theorems \ref{thm2} and \ref{thm1} and \eqref{eq:thm3step2first}, we deduce 
    \begin{align}
    \norm{u_h - \tilde u_h}_{L^\infty(\Omega)} & \le C\norm{u_h-\tilde u_h}_{L^\infty(\partial \Omega)} \nonumber \\
     & \le C\norm{\tilde u - \tilde u_h}_{L^\infty(\partial \Omega)} + C\norm{u- u_h}_{L^\infty(\partial \Omega)} \nonumber \\
     & \le C\norm{\tilde u-\tilde \chi}_{\alpha(h),\widetilde \Omega} + C\norm{\tilde u -\tilde u_h}_{L^2(\widetilde \Omega)} + C\norm{u- u_h}_{L^\infty(\partial \Omega)} \nonumber \\
     & \le C\norm{\tilde u-\tilde \chi}_{\alpha(h),\widetilde \Omega} + C\norm{u- u_h}_{L^\infty(\partial \Omega)}. \label{eq:thm3step3}
    \end{align}
    
    Therefore, using triangle inequality, we obtain
\[
 \norm{u-u_h}_{L^\infty(\Omega)} \le C \norm{u-\chi}_{\alpha(h),\Omega} + C\norm{u-u_h}_{L^\infty(\partial \Omega)}.
\]
\end{proof}

\begin{cor}\label{cor2}
    In addition to the assumption of Theorem  \ref{thm3}, we assume $u \in W^{1+r,\infty}(\Omega)$.
    Then, we have
    \begin{equation}
    \norm{u-u_h}_{L^\infty(\Omega)} \le 
C h^{r}\norm{u}_{W^{1+r,\infty}(\Omega)}.\label{eq:cor2}
    \end{equation}
    \end{cor}

    \begin{proof}
First, in view of the standard interpolation error estimate, we have 
\[
 \inf_{\chi \in V_h}\norm{u-\chi}_{\alpha(h),\Omega} \le C h^r(h+\alpha(h))\norm{u}_{W^{1+r,\infty}(\Omega)}. 
\]
To perform the estimation for $\|u-u_h\|_{L^\infty(\partial\Omega)}$, 
we let $e \in \Eb_h$ and $K \in \T_h$ such that $e \subset \overline K$. 
Moreover, let $\chi\in V_h$ be arbitrary. By the inverse inequality, we have 
\begin{align*}
\norm{u-u_h}_{L^\infty(e)} & \le 
 \norm{u-\chi}_{L^\infty(e)} + \norm{u_h-\chi}_{L^\infty(e)} \\
&\le \norm{u-\chi}_{L^\infty(K)} + Ch_e^{-1/2}\norm{u_h-\chi}_{L^2(e)} .
\end{align*}
Using \eqref{eq:dgconti}, \eqref{eq:dgcoercive}, and \eqref{eq:galerkin}, we have 
$\|\chi-u_h\|_{V^2}\le C\|\chi-u\|_{V^2}$ and, consequently,  
\[
 h_e^{-1/2} \norm{\chi-u_h}_{L^2(e)}\le C\|\chi-u\|_{V^2}. 
\]
Therefore, 
\[
 \norm{u-u_h}_{L^\infty(e)}\le \norm{u-\chi}_{L^\infty(K)}+C\|\chi-u\|_{V^2}.
\]
Choosing $\chi$ as the Lagrange interpolation of $u$, we deduce 
\[
 \norm{u-u_h}_{L^\infty(e)}\le Ch^{r}(h+\alpha(h))|u|_{W^{r+1,\infty}(K)}+Ch^r|u|_{W^{r+1,\infty}(\Omega)}. 
\]
Summing up those estimate, we obtain the desired \eqref{eq:cor2}. 
    \end{proof}

\section{Numerical examples}
\label{s:ne}

In this section, we examine the weak discrete maximum principle (Theorem  \ref{thm2}) and the $L^\infty$ error estimate (Corollary \ref{cor2}) using numerical examples. 
We consider the square domain $\Omega$ (see Fig. \ref{fig:squaremesh}) and the L-shape domain $\Omega$ (see Fig. \ref{fig:lshapemesh}). 


\begin{figure}[bt]
    \centering
       \subfloat[][Square domain]{\includegraphics[scale=0.7]{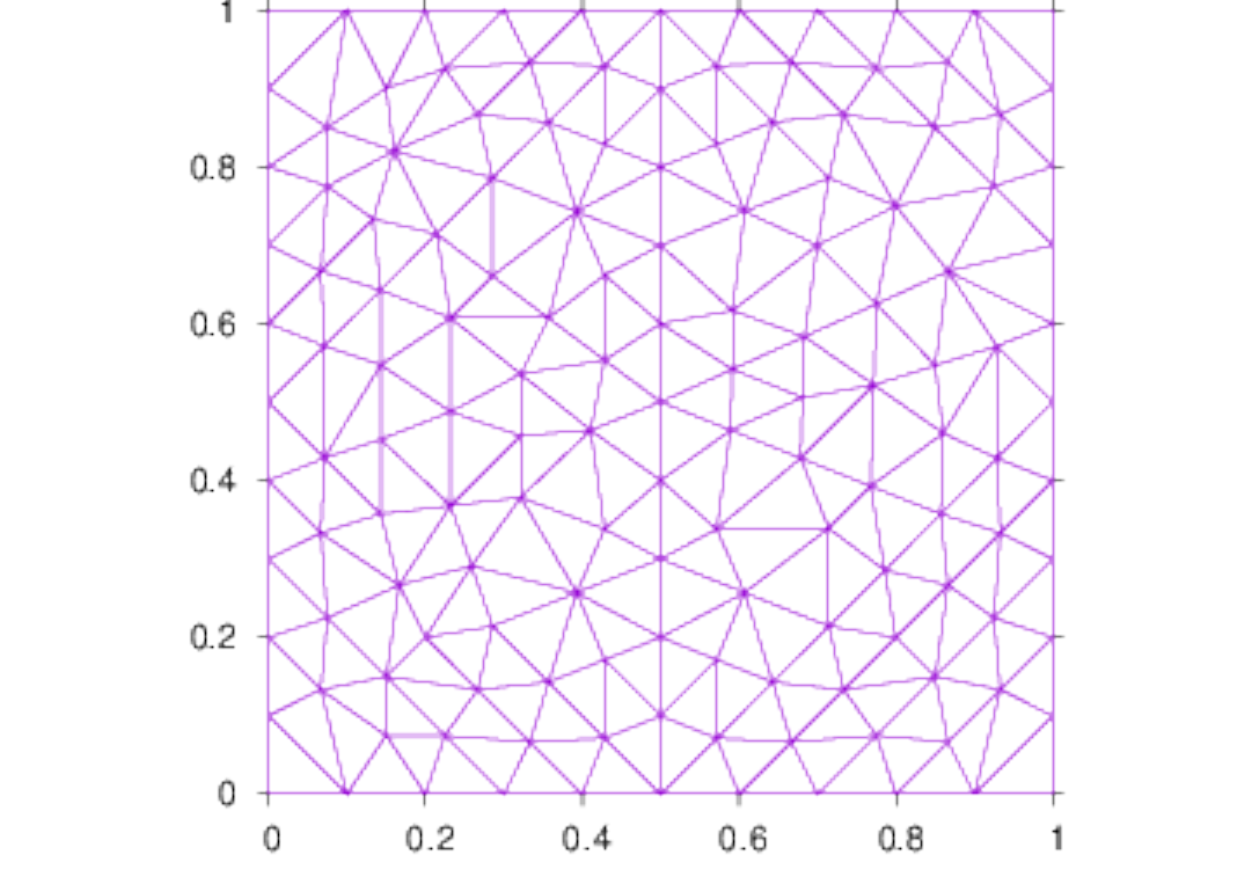}\label{fig:squaremesh}}
       \subfloat[][L-shape domain]{\includegraphics[scale=0.7]{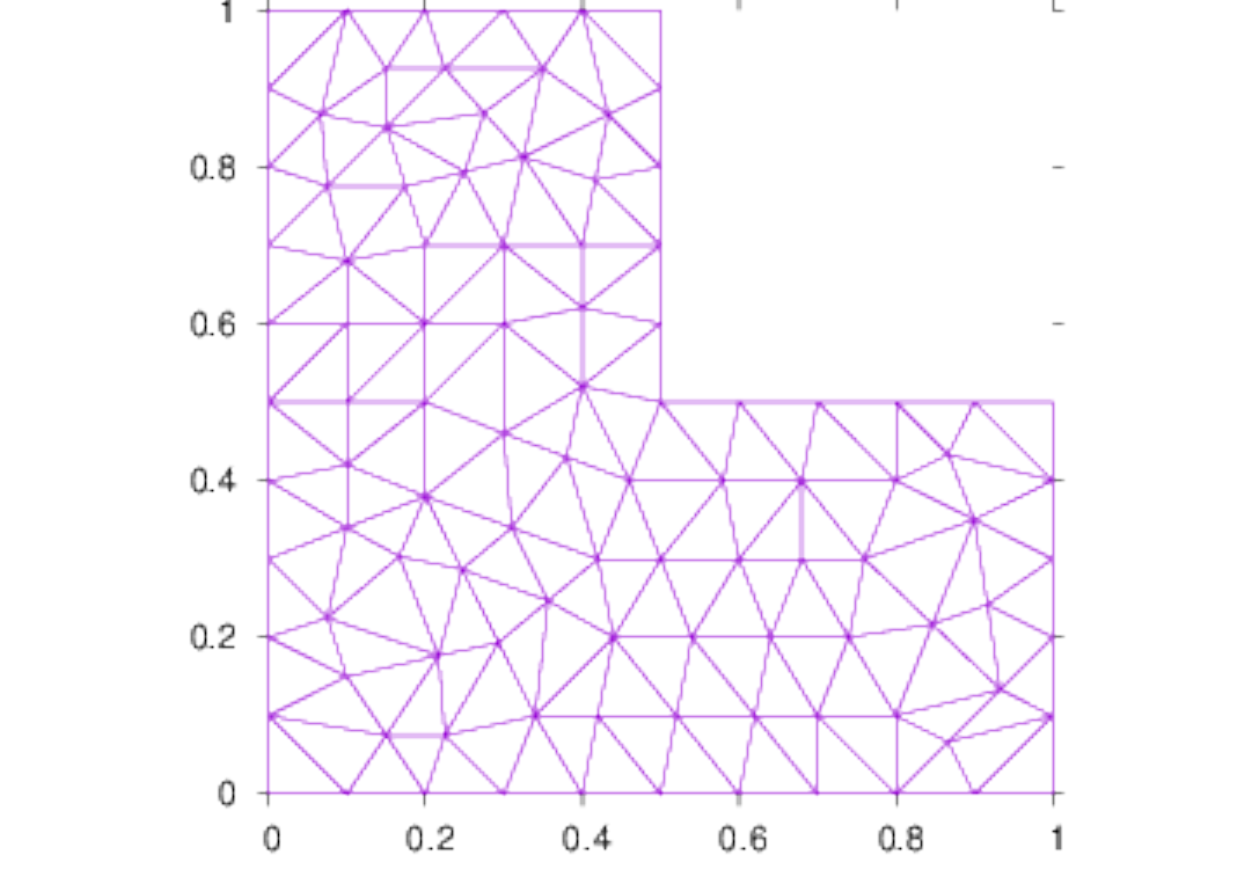}\label{fig:lshapemesh}}
    \caption{Domain $\Omega$}
\end{figure}

First, we solve (DG;$f,g$) with   
$f = 0$ and $g = \cos(\pi x) \cos (\pi y)$; the solution $u_h$ satisfies \eqref{eq:discreteharmonic}.
The minimum and maximum values of $u_h$ on $\Omega$ and $\partial \Omega$ are reported in Tab.~\ref{tb:wdmp}. We see from Tab.~\ref{tb:wdmp} that the minimum and maximum values on $\Omega$ agree with those on $\partial \Omega$. We infer that the discrete maximum principle \eqref{eq:dmp} actually holds with $C=1$.

\begin{table}[bt]
\caption{Minimum and Maximum on $\Omega$ or $\partial \Omega$}
\centering
\begin{tabular}{lc|cc|cc}\hline
Domain & $h$ & $\displaystyle \min_{\Omega}u_h$ & $\displaystyle \min_{\partial \Omega}u_h$ & $\displaystyle \max_{\Omega}u_h$ & $\displaystyle \max_{\partial \Omega}u_h$ \\ \hline\hline
\multirow{2}{*}{Square}  & $0.152069063$ & $-1.01415829$ & $-1.01415829$ & $1.01407799$ & $1.01407799$ \\ \cline{2-6} 
& $0.0762297934$ & $-1.00438815$&$-1.00438815$&$1.00437510$&$1.00437510$ \\ \hline
\multirow{2}{*}{L-shape}  & $0.152069063$ & $-1.01406865$ & $-1.01406865$ & $1.01414407$ & $1.01414407$ \\ \cline{2-6} 
& $0.0790226728$ & $-1.00437424$&$-1.00437424$&$1.00437503$&$1.00437503$ \\ \hline
\end{tabular}
\label{tb:wdmp}
\end{table}

Finally, we consider (BVP;$f,g$)  with $f(x,y)=2\pi^2\sin(\pi x)\sin(\pi y)$ and $g(x,y)=\sin(\pi x)\sin(\pi y)$. The exact solution is given as $u(x,y) = \sin(\pi x)\sin(\pi y)$. 
We examine errors $\|u-u_h\|_{L^\infty(\Omega)}$ with $r=1$ ($\mathcal{P}^1$ element) and $r=2$ ($\mathcal{P}^2$ element). Results are shown in Fig. \ref{fig:p1error} and Fig. \ref{fig:p2error}. We observe that the order is almost $O(h^{1+r})$: the optimal convergence rate is actually observed. This implies that our $L^\infty$ error estimate, Corollary \ref{cor2}, has room for improvement. 

\begin{figure}[bt]
\centering
   \subfloat[][$r=1$]{\includegraphics[scale=0.8]{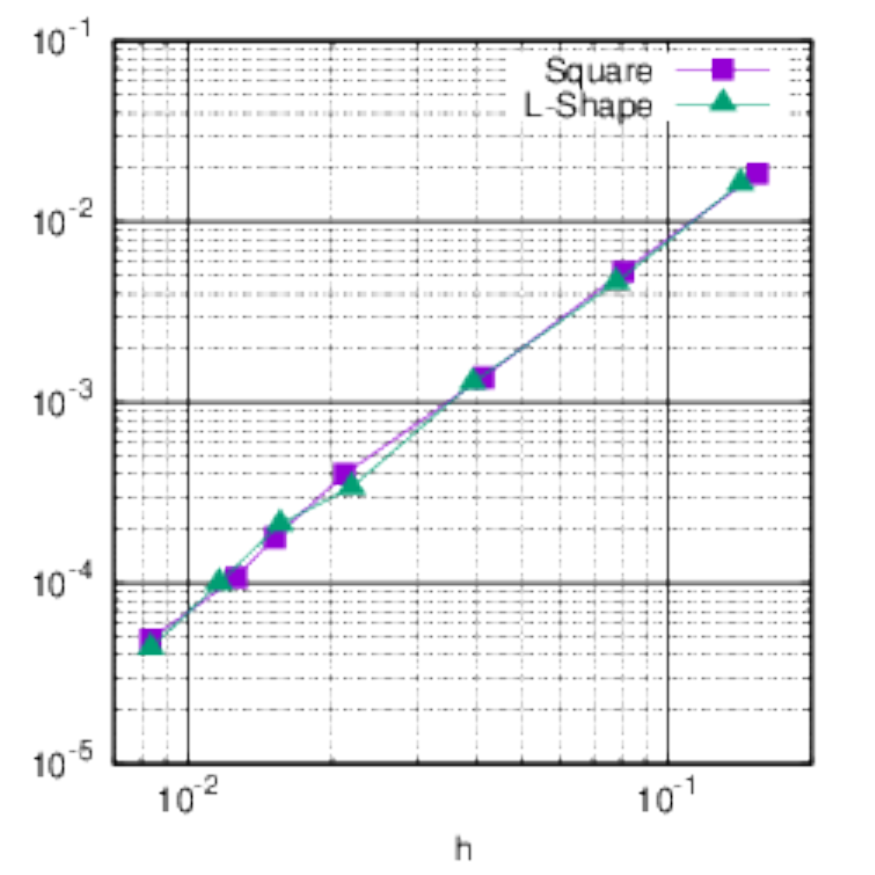}\label{fig:p1error}}
    \subfloat[][$r=2$]{\includegraphics[scale=0.8]{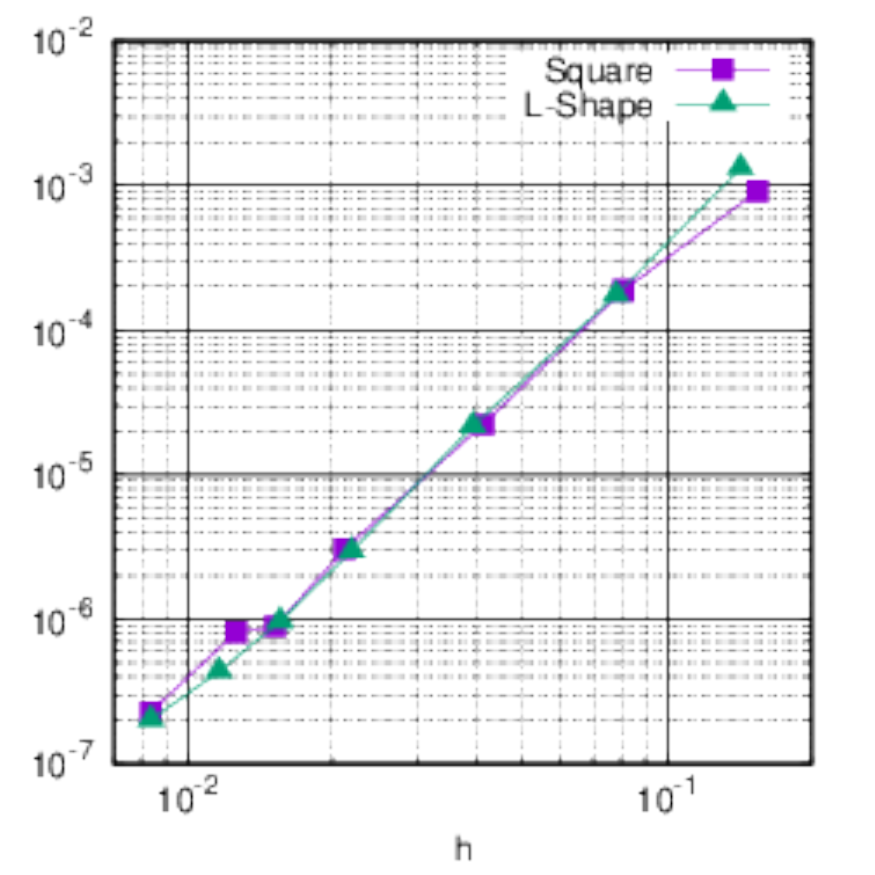}\label{fig:p2error}}
\caption{$L^\infty$ errors $\|u-u_h\|_{L^\infty(\Omega)}$}
\end{figure}

\section{Conclusion}
We have shown the interior error estimate and discrete weak maximum principle of the DG method for the Poisson equation. 
Those results are extensions of the standard FEM \cite{MR551291} to the DG method.
Moreover, we have derived the $L^\infty$ error estimate as an application of the discrete weak maximum principle. 
Unfortunately, our $L^\infty$ error estimate is only sub-optimal. 
The optimal rate has been proved in \cite{MR2113680} by another method. 
This implies that we need to deep consider the imposition of the Dirichlet boundary condition in the DG method. In particular, we will study more precise estimates of $\alpha(h)$ and $\norm{u-u_h}_{L^\infty(\partial \Omega)}$ in the future works.

\section*{Acknowledgment}
The first author was supported by Program for Leading Graduate Schools, MEXT, Japan.
The second author was supported by JST CREST Grant Number JPMJCR15D1,
Japan, and JSPS KAKENHI Grant Number 15H03635, Japan.


\bibliographystyle{spmpsci}
\bibliography{bibliography}
%
%
%
%
\end{document}